\documentclass{article}
\usepackage{graphicx} 
\usepackage[margin=1in]{geometry}
\title{Sinkhorn algorithms and linear programming solvers for optimal partial transport problems}
\author{Yikun Bai}
\date{June 2024}
\usepackage{amsmath, amssymb}
\usepackage{amsmath}
\usepackage{amsthm}  

\newtheorem{remark}{Remark}
\newtheorem{lemma}{Lemma}
\newtheorem{example}{Example}
\newtheorem{theorem}{Theorem}

\usepackage[linesnumbered,ruled,vlined]{algorithm2e}

\usepackage{xcolor}
\newcommand\mycommfont[1]{\footnotesize\ttfamily\textcolor{blue}{#1}}
\SetCommentSty{mycommfont}
\SetKwInput{KwInput}{Input}                
\SetKwInput{KwOutput}{Output}   

\begin{document}

\maketitle

\begin{abstract}
In this note, we generalize the classical optimal partial transport (OPT) problem by modifying the mass destruction/creation term to function-based terms, introducing what we term ``generalized optimal partial transport'' problems. We then discuss the dual formulation of these problems and the associated Sinkhorn solver. Finally, we explore how these new OPT problems relate to classical optimal transport (OT) problems and introduce a linear programming solver tailored for these generalized scenarios.
\end{abstract}

\section{Notation and assumptions}
\begin{itemize}
\item $\mathbb{R}=(-\infty,\infty), \mathbb{R}_+=[0,\infty), \mathbb{R}_{++}=(0,\infty)$.

\item $\Omega$: A non-empty convex open subset of $\mathbb{R}^d$ for some $d\in \mathbb{N}$.

\item $c: \Omega^2\to \mathbb{R}_+$: A lower semi-continuous function, referred to as the ``\textit{cost function}'' in this note. 

Given a topological set $E$, $f: E\to \mathbb{R}$ is said to be lower-semi continuous at $x\in E$ if the following holds: 

Pick $x_n\to x$ (where the convergence is determined by the topology of $E$), then 
$$\liminf_{n}f(x_n)\ge f(x). $$

\item $\mathcal{P}(\Omega)$: The set of all probability measures defined on $\Omega$. 

\item $\mathcal{M}_+(\Omega)$: The set of all positive Radon measures defined on $\Omega$.

\item $\mathcal{M}(\Omega)$: The set of all Radon measures defined on $\Omega$. 
\item $C_0(\Omega):$ set of all $C_0$ functions. In particular, it is defined as  
$$C_0(\Omega):=\{f: \Omega\to \mathbb{R}, f\text{ is continuous},\exists \text{compact set} K\subset\Omega, s.t. f(x)=0,\forall x\notin K\}.$$

\item $\gamma$: A element in $\mathcal{P}(\Omega^2)$ or $\mathcal{M}_+(\Omega^2)$. 

\item $\mu, \nu$: elements in $\mathcal{P}(\Omega)$ or $\mathcal{M}_+(\Omega)$. 

\item $\pi_1, \pi_2: \Omega^2\to \Omega$: Canonical projections defined on $\Omega^2$. Specifically, 
$\pi_1((x,y)) = x$ and $\pi_2((x,y)) = y$. 

\item $f_\#\mu$: The push-forward of measure $\mu$ by a function $f: \Omega\to \Omega$. Specifically, 
$f_\#\mu(A) = \mu(f^{-1}(A))$ for any Borel set $A$.

\item $\gamma_1, \gamma_2$: $\gamma_1 = (\pi_1)_\#\mu$ is the first marginal of $\gamma$. $\gamma_2$ is defined similarly.

\item $\mu = \nu$: Indicates that two measures are equal. For each Borel set $A$, $\mu(A) = \nu(A)$, or equivalently, for each test function $\phi\in C_0$, we have 
$$\int\phi d\mu =\int \phi d\nu. $$

\item $\mu \leq \nu$: Indicates that $\mu$ is dominated by $\nu$. Specifically, for each Borel set $A$, $\mu(A) \leq \nu(A)$. Equivaletly, for each test function $\phi \in C_0(\Omega)$, $\phi\ge0$, we have 
$$\int \phi d\mu\leq \int \phi d\nu.$$

\item $\Gamma(\mu,\nu):=\{\gamma \in \mathcal{P}(\Omega^2) : \gamma_1 = \mu, \gamma_2 = \nu\}$. 

\item $\Gamma_\leq(\mu,\nu) := \{\gamma \in \mathcal{M}_+(\Omega^2) : \gamma_1 \leq \mu, \gamma_2 \leq \nu\}$.

\item $\Gamma_\leq^\eta(\mu,\nu) := \{\gamma \in \mathcal{M}_+(\Omega^2) : \gamma_1 \leq \mu, \gamma_2 \leq \nu, |\gamma| = \eta\}$, where $\eta \in [0, \min\{|\mu|, |\nu|\}]$. 

\item $\text{Supp}(\mu), \text{Supp}(\nu)$: The support sets of $\mu$ and $\nu$. 

\item $\mathcal{L}$: The Lebesgue measure defined on $\Omega$. 

\item Continuous setting: In this case, we suppose $\mu = f d\mathcal{L}$ and $\nu = g d\mathcal{L}$, where $f, g \in L_1(\Omega)$. 

\item Discrete setting: In this case, we suppose $\mu = \sum_{i=1}^n p_i\delta_{x_i}$ and $\nu = \sum_{j=1}^m q_j\delta_{y_j}$, where $n, m \in \mathbb{N}$ and $p_i, q_j > 0$ for all $i, j$. We denote $X = \{x_i : i \in [1:n]\}$ and $Y = \{y_j : j \in [1:m]\}$.

\item $dx, dy$: $dx$ is a reference measure in $\Omega$ such that $\mu \ll dx$, and similarly, $dy$ is a reference measure in $\Omega$ such that $\nu \ll dy$. Both are set to be probability measures.

\textit{Typically, we assume $d\mu/dx, d\nu/dy$ are $L_1$ functions.} 

\item $dxdy = dx \otimes dy$: The product measure of $dx$ and $dy$, defined as the reference measure on $\Omega^2$.

\item $\frac{d\mu}{dx}, \frac{d\nu}{dy}$: Density functions of $\mu$ and $\nu$ with respect to $dx$ and $dy$, respectively.

\item Inner product and integration: In this note, we use the following identity: 
$$\int f d\mu = \mu(f) = \langle f, \mu \rangle,$$
where $f$ is a measurable function defined on $\Omega$.

\item $\|\cdot\|_{TV}$: Total variation norm defined by:
$$\|\mu\|_{TV} = \int d|\mu|,$$
where $|\mu| = \mu_+ + \mu_-$, and $\mu_+, \mu_-$ are the unique measure decomposition of the signed measure $\mu \in \mathcal{M}(\Omega)$.

\item $\mu = \frac{d\mu}{d\nu} \nu + \mu^\perp$: The Lebesgue decomposition of $\mu$. By classical measure theory, this decomposition is unique.

\item $\mu \ll \nu$: Indicates that $\mu$ is absolutely continuous with respect to $\nu$.

\item $\|\mu\|, \|\nu\|$: Total masses of $\mu$ and $\nu$, respectively. $\|p\|, \|q\|$ are defined similarly.

\item $KL(\mu\parallel \nu)$: The Kullback-Leibler (KL) divergence between $\mu$ and $\nu$. See \eqref{eq:kl}. 

\item $\|\mu-\nu\|_{TV} = \|\mu - \nu\|_{TV}, P\|\mu-\nu\|_{TV} = \|\mu - \nu\|_{PTV}$: Total variation and partial total variation between $\mu$ and $\nu$. See \eqref{eq:tv} and \eqref{eq:ptv}. 

\item $f, D_f$: $f$ is an entropy function, and $D_f$ is the corresponding $f$-divergence.

\item $f'_\infty = \lim_{x \to \infty} \frac{f(x)}{x} \in \mathbb{R} \cup \{\infty\}$: The growth rate of $f$ at infinity.

\item $\iota_{A}$, where $A\in\Omega$: indicator function such that $\iota_A(x)=0$ if $x\in A$ and $\iota_A(x)=\infty$ elsewhere. 

\item $(E, E^*)$: $E$ is a convex set, and $E^*$ is its dual space. The dual pairing is defined by:
$$(x, x^*) \mapsto x^*(x) = \langle x, x^* \rangle \in \mathbb{R}.$$

\item $\lambda > 0$: A constant used in the OPT problem \eqref{eq:opt_1}. 

\item $\lambda_1, \lambda_2: \Omega \to \mathbb{R}_+$: Positive bounded functions. 

\item $\eta \in [0, \min\{|\mu|, |\nu|\}]$: A constant used in the MOPT problem \eqref{eq:mopt}. 

\item $\epsilon > 0$: A constant used in the entropic OPT problem \eqref{eq:egopt}.

\item $K = e^{-c/\epsilon}$

\item $\mu|_A$, where $A \subset \Omega$: The restriction of the measure $\mu$ on set $A$. Specifically, 
$$\mu|_A(B) = \mu(B \cap A)$$ for all Borel sets $B \subset \Omega$.

\item $\phi|_A$, where $A \subset \Omega$: The restriction of the function $\phi$ on $A$. Specifically, 
$$\phi|_A(x) = \begin{cases}
    \phi(x) & \text{if } x \in A \\
    0       & \text{elsewhere}
\end{cases}.$$
Note, the above definition can be equivalently redefined as 
$$\Phi\mid_A:A\to \mathbb{R}$$
with $\Phi\mid_A(x)=\Phi(x),\forall x\in A$. In this note, for conveneince, we to not distinguish these two formualations. 

\item $x\odot y,\frac{x}{y}$ where $x,y$ are vectors or matrices with same shape:  Point-wise multiplication/division.  

\item $1_n$: a $n\times 1$ vector where entry is $1$. Similarly, we can define $1_m, 1_{n\times m}, 0_n,0_m\ldots$. 

\end{itemize}

\section{Outlines}\label{sec:outline}
In section \ref{sec:background}, we introduce the foundational concepts of $f$-divergence, unbalanced optimal transport, and various optimal partial transport problems. Section \ref{sec:opt} presents the formulations of new optimal partial transport problems, their entropic versions, and the corresponding dual forms. Section \ref{sec:sinkhorn} discusses the Sinkhorn algorithms tailored for these new OPT problems. In section \ref{sec:lp}, we explore the relationship between the generalized OPT problems and classical OT problems, including discussions on the linear programming solver adapted for generalized OPT problems. Section \ref{sec:mopt} is dedicated to examining both the Sinkhorn and linear programming solvers for mass-constrained OPT problems.

\section{Background}\label{sec:background}
\subsection{Entropy Function and $f$-Divergence}
An \textbf{entropy function} $f: \mathbb{R} \to \mathbb{R} \cup \{\infty\}$ is lower semi-continuous, convex, and its domain $D(f) := \{x : f(x) < \infty\}$ is contained within $[0, \infty)$ and intersects $(0, \infty)$. The growth rate of $f$ at infinity is defined as:
$$f'_\infty = \lim_{x \to \infty} \frac{f(x)}{x} \in \mathbb{R} \cup \{\infty\}.$$

Given measures $\mu, \nu \in \mathcal{M}(\Omega)$ with the Lebesgue decomposition $\mu = \frac{d\mu}{d\nu}\nu + \mu^\perp$, the \textbf{$f$-divergence} is defined by:
\begin{align}
D_f(\mu\parallel \nu) := \int_{\Omega} f\left(\frac{d\mu}{d\nu}\right) d\nu + f'_\infty |\mu^\perp|, \label{eq:D_f}
\end{align}
assuming the convention $0 \cdot \infty = 0$.
\begin{example}\label{ex:entropy_func}
We consider the following examples of entropy functions:
\begin{itemize}
    \item \textbf{Kullback-Leibler (KL) Divergence:}
    \begin{align}
      f(s) = f_{KL}(s) := \begin{cases}
    s \ln(s) - s + 1 & \text{if } s > 0 \\
    1                & \text{if } s = 0 \\ 
    \infty           & \text{if } s < 0 
    \end{cases}. \label{eq:f_KL}
    \end{align}
    Here, $(f_{KL})'_\infty = \infty$, and the KL divergence is:
    \begin{align}
     KL(\mu\parallel \nu) := \begin{cases}
         \int \ln\left(\frac{d\mu}{d\nu}\right) \frac{d\mu}{d\nu} d\nu - \|\mu\| + \|\nu\| & \text{if } \mu \ll \nu \\
         \infty & \text{otherwise}
     \end{cases}, \label{eq:kl}
    \end{align}
    with $\ln(0) \cdot 0 = 0$ assumed.
    \item \textbf{Total Variation:}
    \begin{align}
    f(s) = f_{TV}(s) := \begin{cases}
    |s - 1| & \text{if } s \geq 0 \\
    \infty & \text{if } s < 0 
    \end{cases}, \label{eq:f_tv}
    \end{align}
    where $(f_{TV})'_\infty = 1$, and the total variation divergence is:
    \begin{align}
 \|\mu - \nu\|_{TV} := \int_{\Omega} d|\mu - \nu|, \label{eq:tv}
    \end{align}
    \item \textbf{Partial Total Variation:}
    \begin{align}
    f(s) = f_{PTV}(s) := \begin{cases}
    1 - s & \text{if } s \in [0, 1] \\
    \infty & \text{if } s < 0 \text{ or } s > 1
    \end{cases}, \label{eq:f_ptv}
    \end{align}
    with $(f_{PTV})'_\infty = \infty$. The corresponding divergence is termed "partial total variation divergence":
    \begin{align}
    \|\mu-\nu\|_{PTV} := \begin{cases}
    \|\nu - \mu\|_{TV} = \|\nu\| - \|\mu\| & \text{if } \mu \leq \nu \\
    \infty & \text{otherwise}
    \end{cases}, \label{eq:ptv}
    \end{align}
    \item \textbf{Equality Constraint Divergence:}
    \begin{align}
    f(s) = \iota_{\{1\}}(s) := \begin{cases}
    0      & \text{if } s = 1 \\
    \infty & \text{otherwise}
    \end{cases}, \label{eq:f_iota}
    \end{align}
    where $(\iota_{\{1\}})'_\infty = \infty$. The associated divergence is termed \textbf{equality-constraint divergence}:
    \begin{align}
    \iota_{=}(\mu\parallel \nu) := \begin{cases}
    0 & \text{if } \mu = \nu \\
    \infty & \text{otherwise}
    \end{cases}, \label{eq:D_=}
    \end{align}
    \item \textbf{Non-constraint Divergence Function:}
    \begin{align}
    f(s) = \mathbf{0}(s) := 0,
    \end{align}
    with $(\mathbf{0})'_\infty = 0$, the associate divergence, termed ``$D_\bold{0}$'', is:
    \begin{align}
    D_\bold{0}(\mu \parallel\nu) = 0, \label{eq:D_+}
    \end{align}
    Note, $\mathbf{0}$ does not meet classical entropy function conditions as it does not ensure $D(\mathbf{0}) \subset [0, \infty)$.
    \item \textbf{Partial-Transportation Constraint Divergence Function:}
    \begin{align}
    f(s) = \iota_{[0,1]}(s) = \begin{cases}
    0 & \text{if } s \in [0, 1] \\
    \infty & \text{if } s < 0 \text{ or } s > 1
    \end{cases}, \label{eq:f_<=}
    \end{align}
    where $(\iota_{[0,1]})_\infty = \infty$. The corresponding divergence, termed ``$\iota_\leq$'', is:
    \begin{align}
    \iota_\leq(\mu\parallel\nu) = \begin{cases}
    0 &      \text{if } \mu \leq \nu \\
    \infty & \text{otherwise}
    \end{cases}.
    \end{align}
\end{itemize}
\end{example}

\begin{remark}
It is straightforward to verify, $f_{PTV}=f_{TV}+\iota_{[0,1]}$, and thus,
$$PTV(\mu\parallel \nu )=\|\mu-\nu\|_{TV}+\iota_\leq(\mu\parallel \nu).$$
\end{remark}

\subsection{Optimal Transport and Optimal Partial Transport}
Optimal transport (OT) problems, as studied in \cite{villani2021topics,Villani2009Optimal}, seek the most cost-efficient way to transfer mass between two probability measures. Specifically, given $\mu, \nu \in \mathcal{P}(\Omega)$, the OT problem is defined as:
\begin{align}
OT(\mu, \nu) := \min_{\gamma \in \Gamma(\mu, \nu)} \int_{\Omega^2} c(x, y)  d\gamma(x, y) = \min_{\gamma \in \Gamma(\mu, \nu)} \langle c, \gamma \rangle \label{eq:ot}
\end{align}

Classical OT theory ensures the existence of a minimizer, as $c$ is lower semi-continuous, allowing the substitution of $\inf$ with $\min$.

In a discrete setting, where $\mu = \sum_{i=1}^n p_i \delta_{x_i}$ and $\nu = \sum_{j=1}^m q_j \delta_{y_j}$, which are discrete probability measures, the problem becomes:
\begin{align}
OT(\mu, \nu) := \min_{\gamma \in \Gamma(p, q)} \langle c, \gamma \rangle \label{eq:ot_discrete}
\end{align}
Here, $c \in \mathbb{R}^{n \times m}$ with $c_{i,j} = c(x_i, y_j)$ for all $i, j$, and $\Gamma(p, q) = \{\gamma\in \mathbb{R}_+^{n\times m}: \gamma 1_m=p, \gamma^\top 1_n=q\}$.

The problem \eqref{eq:ot}, \eqref{eq:ot_discrete} is referred to as the \textbf{balanced optimal transport problem} because it requires $\|\mu\| = \|\nu\|$. In an unbalanced setting, where $\mu, \nu \in \mathcal{M}_+(\Omega)$ may not be equal, the problem \eqref{eq:ot} can be generalized as follows:
\begin{align}
UOT(\mu, \nu; D_{f_1}, D_{f_2}) := \inf_{\gamma \in \mathcal{M}_+(\Omega^2)} \langle c, \gamma \rangle + D_{f_1}(\gamma_1\parallel \mu) + D_{f_2}(\gamma_2\parallel \nu) \label{eq:uot},
\end{align}
where $D_{f_1}, D_{f_2}$ are $f$-divergences defined in \eqref{eq:D_f}.

\begin{remark}
When $D_{f_1}, D_{f_2} = \iota_{=}$, the above formulation reverts to the balanced OT \eqref{eq:ot} when $\|\mu\| = \|\nu\|$, and to $\infty$ otherwise. Thus, \eqref{eq:ot} can be redefined as:
\begin{align}
OT(\mu, \nu) := \inf_{\gamma \in \mathcal{M}_+(\Omega^2)} \langle c, \gamma \rangle + \iota_{=}(\gamma_1\parallel \mu) + \iota_{=}(\gamma_2\parallel \nu) = \begin{cases}
    \min_{\gamma \in \Gamma(\mu, \nu)} \langle c, \gamma \rangle & \text{if } \|\mu\| = \|\nu\| \\
    \infty & \text{otherwise}
\end{cases} \label{eq:ot_2}.
\end{align}
\end{remark}

A specific class of unbalanced optimal transport, termed \textbf{optimal partial transport} and discussed in \cite{caffarelli2010free, Piccoli2014Generalized, figalli2010optimal, figalli2010new}, is formulated by setting $D_{f_1}, D_{f_2} = \lambda TV$:
\begin{align}
OPT(\mu, \nu; \lambda) &:= \inf_{\gamma \in \mathcal{M}_+(\Omega^2)} \langle c, \gamma \rangle + \lambda \left( TV(\gamma_1\parallel \mu)+TV(\gamma_2\parallel \nu)  \right) \label{eq:opt_1} \\
&= \inf_{\gamma \in \gamma_\leq(\mu, \nu)} \langle c, \gamma \rangle + \lambda (\|\mu - \gamma_1\| + \|\nu - \gamma_2\|) \label{eq:opt_2} \\
&= \inf_{\gamma \in \mathcal{M}_+(\Omega^2)} \langle c, \gamma \rangle + \lambda (PTV(\gamma_1\parallel \mu) + PTV(\gamma_2\parallel\nu)) \label{eq:opt_3}
\end{align}
where 
$$\Gamma_\leq(\mu, \nu) = \{\gamma \in \mathcal{M}_+(\Omega^2) : \gamma_1 \leq \mu, \gamma_2 \leq \nu\};$$ 
\eqref{eq:opt_2} has been validated by \cite{Piccoli2014Generalized}, and \eqref{eq:opt_3} directly follows from the definition of PTV (see \eqref{eq:ptv}).

Additionally, an equivalent formulation of \eqref{eq:opt_2}, termed ``\textbf{mass-constraint optimal partial transport}'', is proposed in this paper:
\begin{align}
MOPT_\eta(\mu, \nu) := \inf_{\gamma \in \Gamma_{\leq}^\eta(\mu,\nu)} \langle c, \gamma \rangle \label{eq:mopt} 
\end{align}
where $\eta \in [0, \min\{|\mu|, |\nu|\}]$, 
$$ \inf_{\gamma \in \Gamma_{\leq}^\eta}(\mu,\nu)=\{\gamma\in \mathcal{M}_+(\Omega^2):\gamma_1\leq \mu,\gamma_2\leq \nu,\|\gamma\|=\eta\}.$$

Optimal partial transport problems \eqref{eq:opt_1}, \eqref{eq:mopt} admit a minimizer.

\section{A Generalized Version of the OPT Problem}\label{sec:opt}
In this section, we assume $\lambda_1, \lambda_2: \Omega \to \mathbb{R}_+$ are measurable, bounded functions. We introduce the generalized OPT problem defined as:
\begin{align}
GOPT(\mu, \nu; \lambda_1, \lambda_2) := \inf_{\gamma \in \mathcal{M}_+(\Omega^2)} \langle c, \gamma \rangle + 
\begin{cases}
TV(\lambda_1\gamma_1\parallel \lambda_1\mu)\\
\text{or}\\
\overline{PTV}(\lambda_1\gamma_1 \parallel \lambda_1\mu)
\end{cases} + 
\begin{cases}
TV(\lambda_2\gamma_2 \parallel \lambda_2\nu)\\
\text{or}\\
\overline{PTV}(\lambda_2\gamma_2 \parallel \lambda_2\nu)
\end{cases}\label{eq:gopt}
\end{align}
where $\overline{PTV}(\lambda_1\gamma_1 \parallel \lambda_1\mu) = PTV(\lambda_1\gamma_1 \parallel \lambda_1\mu) + \iota_\leq(\gamma_1\parallel \mu)$, and $\overline{PTV}(\lambda_2\gamma_2 \parallel \lambda_2\nu)$ is defined similarly.

We refer to $TV(\lambda_1\gamma_1 \parallel \lambda_1\mu)$ (or similarly for $\overline{PTV}$) as the ``\textit{first penalty term}'' and $TV(\lambda_2\gamma_2 \parallel \lambda_2\nu)$ (and similarly for $\overline{PTV}$) as the ``\textit{second penalty term}''. Intuitively, $\lambda_1(x)$ models the penalty of mass destruction/creation on the source measure $\mu$ at point $x$, $\lambda_2(y)$ models the penalty of mass destruction/creation on $\nu$ at point $y$. 

\begin{remark}
When $\lambda_1 = \lambda_2 \equiv \lambda$, this formulation reduces to the classical OPT problem as defined in \eqref{eq:opt_1}.
\end{remark}

\begin{remark}\label{rk:partial_transport}
Unlike the classical OPT problem \eqref{eq:opt_1} and \eqref{eq:opt_2}, in this formulation, when selecting $TV(\lambda_1\gamma_1 \parallel \lambda_1\mu)$ and $TV(\lambda_2\gamma_2 \parallel \lambda_2\nu)$, we are not restricted to searching within $\Gamma_\leq(\mu, \nu)$ in $\mathcal{M}_+(\Omega^2)$. For instance, consider $\mu = \delta_0$ and $\nu = \delta_0 + \delta_1$ with $\lambda_1 \equiv 0$ and $\lambda_2 \equiv 100$. The optimal transportation plan, denoted by $\gamma$, would be $\delta_{(0,0)} + \delta_{(1,1)}$, yielding a zero cost in \eqref{eq:gopt}. However, $\gamma_1 = \delta_0 + \delta_1 \not\leq \delta_0 = \mu$. 
\end{remark}
\textit{In the rest of this note, we will use $\begin{cases}
    A\\ 
    B
\end{cases}$ to denote the $\begin{cases}
    A\\ 
    \text{or }\\
    B\end{cases}$ for simplicity.}

\subsection{Entropic Regularization}
Given $\gamma \in \mathcal{M}_+(\Omega^2)$, if $\gamma \ll dxdy$, the negative entropy term is defined as:
\begin{align}
-H(\gamma) &:= \int_{\Omega^2} \left(\ln\left(\frac{d\gamma}{dxdy}\right)-1\right) \frac{d\gamma}{dxdy}dxdy = KL(\gamma \parallel dxdy) - 1\label{eq:H(gamma)}.
\end{align}
If $\gamma \not\ll dxdy$, we emphasize the difference by still using $\frac{d\gamma}{dxdy} dxdy$ in the definition to indicate that $\frac{d\gamma}{dxdy} dxdy \neq d\gamma$ in this case.

With a chosen $\epsilon > 0$, we define the entropic (generalized) optimal transport problems based on \eqref{eq:gopt}:
\begin{align}
EGOPT(\mu, \nu; \lambda_1, \lambda_2, \epsilon) &:= \inf_{\substack{\gamma \in \mathcal{M}_+(\Omega^2) \\ \gamma \ll dxdy}} \langle c, \gamma \rangle - \epsilon H(\gamma) + \begin{cases}
TV(\lambda_1\gamma_1 \parallel \lambda_1\mu) \\
PTV(\lambda_1\gamma_1 \parallel \lambda_1\mu)
\end{cases}+ \begin{cases}
TV(\lambda_2\gamma_2 \parallel \lambda_2\nu) \\
PTV(\lambda_2\gamma_2 \parallel \lambda_2\nu)
\end{cases} \nonumber\\ 
&=\inf_{\substack{\gamma \in \mathcal{M}_+(\Omega^2) \\ \gamma \ll dxdy}} -\epsilon KL(\gamma\parallel K) + \begin{cases}
TV(\lambda_1\gamma_1 \parallel \lambda_1\mu) \\
PTV(\lambda_1\gamma_1 \parallel \lambda_1\mu)
\end{cases}+ \begin{cases}
TV(\lambda_2\gamma_2 \parallel \lambda_2\nu) \\
PTV(\lambda_2\gamma_2 \parallel \lambda_2\nu)
\end{cases}+\text{constant}\label{eq:egopt}.
\end{align}

where in \eqref{eq:egopt}, $K$ is defined as $K := e^{-c/\epsilon}$. The equation follows from the fact:
\begin{align}
KL(\gamma | Kdxdy) &= \int \left(\ln\left(\frac{d\gamma}{Kdxdy}\right) - 1\right) \frac{d\gamma}{dxdy}dxdy + \|Kdxdy\| \nonumber \\
&= \int \left(\ln\left(\frac{d\gamma}{dxdy}\right) - 1\right) \frac{d\gamma}{dxdy}dxdy + \int \frac{c(x,y)}{\epsilon}d\gamma + \|Kdxdy\| \nonumber \\
&= \frac{1}{\epsilon} \left(-H(\gamma) + \langle c, \gamma \rangle \right) + \|Kdxdy\|.
\end{align}
As $\|Kdxdy\|$ is a non-negative constant, independent of $\gamma$, it is disregarded in further discussions.

\begin{remark}
Rigorously speaking, whether in the general entropic unbalanced OT case or in the entropic GOPT, setting $\epsilon = 0$ does not guarantee equivalence between the entropic and the original versions since the optimal solution in the original version might not satisfy $\gamma \ll dxdy$. The convergence of the minimizer of EGOPT \eqref{eq:egopt} to a minimizer in GOPT \eqref{eq:gopt} under general conditions remains an open question.

However, if a minimizer in \eqref{eq:gopt} is absolutely continuous with respect to $dxdy$ and $\Omega$ is finite, as $\epsilon \to 0$, the minimizer of EGOPT converges to this minimizer of GOPT, as shown in \cite[Proposition 3.1]{chizat2018scaling}.
\end{remark}

\subsection{Dual formulation of entropic (generalized) OPT problems}
In this section, we introduce the dual formulation of the entropic (generalized) OPT problem defined in \eqref{eq:egopt}, which equal the original problems (up to a constant): 
\begin{align}
&Dual-EGOPT(\mu,\nu;\lambda_1,\lambda_2)\nonumber\\
&=\sup_{\substack{\phi\in L^\infty(\Omega,dx)\\
\psi\in L^\infty(\Omega,dy)}}
-\epsilon\langle (e^{\frac{\phi\oplus \psi}{\epsilon}}-1)K,dxdy\rangle+
\begin{cases}
\langle \min(\phi,\lambda_1),\mu\rangle+\iota_{\phi\in\mathcal{S}(\lambda_1,dx)}\\
\langle \min(\phi,\lambda_1),\mu\rangle
\end{cases}+
\begin{cases}
\langle \min(\psi,\lambda_2),\nu\rangle+\iota_{\psi\in\mathcal{S}(\lambda_2,dy)}\\
\langle \min(\psi,\lambda_2),\nu\rangle
\end{cases}\label{eq:egopt_dual}
\end{align}
where 
$$\mathcal{S}(\lambda_1,dx):=\{\phi:\phi\ge -\lambda_1,dx-a.s.; \phi=0, dx\mid_{\lambda_1=0}-a.s.\}$$
and $\mathcal{S}(\lambda_2,dy)$ is defined similarly. 

The rest of this section is the derivation of this formulations (in particular, see Remark \ref{rk:dual_form}) and related basic concepts. The main idea is based on \cite[Theorme 3.1]{chizat2018scaling}  and related concepts. 

We first setup some notations and settings. Since $\gamma\ll dxdy$ in EGOPT problem. We have $\gamma_1\ll dx$ and $\gamma_2\ll dy$. Thus we can write $\gamma_1=\frac{\gamma_1}{dx} dx,\gamma_2=\frac{d\gamma_2}{dy}dy$ and can use $r=\frac{\gamma_1}{dx}$ to represent $\gamma_1$, where $r\in L_1(\Omega,dx)$ and similar to $\gamma_2$.

\subsubsection{Background: convex function, conjugate, dual formulation.}
Let $(E,E^*)$ be a pair of convex Hausdorff topologically paired vector spaces, so that the elements of each space can be identified with a continuous linear functional on the other. We refer to \cite{rockafellar1967duality} for details. The topological pairing is the bi-linear form: 
\begin{align}
 (E,E^*) \ni (s,s') \mapsto \langle s,s' \rangle\in \mathbb{R} \label{eq:inner_product} 
\end{align}

\begin{example}
$(L_1(\Omega,dx),L_\infty(\Omega,dx))$ is such a topologically paired vector space. In particular, we have: 
$(L_1(\Omega,dx))^* = L_\infty(\Omega,dx)$ and $(L_\infty(\Omega,dx))^* \supset L_1(\Omega,dx)$. Additionally, the topological pairing is defined by:
$$\langle f,g \rangle = \int_{\Omega} fg  dx $$
\end{example}

\begin{example}
$(\mathbb{R}^d, \mathbb{R}^d)$ is another instance of topologically paired vector spaces. Specifically, $(\mathbb{R}^d)^* = \mathbb{R}^d$. The topological pairing can be described by:
$$\langle s, s' \rangle = \sum_{i=1}^d s_i s'_i$$
\end{example}

A function $f: E \to \bar{\mathbb{R}} = [-\infty, \infty]$ is termed a convex function if it satisfies:
\begin{align}
    f((1-\alpha)x + \alpha y) \leq (1-\alpha)f(x) + \alpha f(y) \label{eq:convex}
\end{align}
for all $x, y \in E$ and $\alpha \in [0,1]$. The domain of $f$, denoted $D(f)$, is:
$$D(f) = \{x \in E \mid f(x) < \infty\}.$$
$f$ is a proper function if $D(f) \cap \{f > -\infty\} \neq \emptyset$.

The convex conjugate of a function $f: E \to [-\infty, \infty]$, denoted as $f^*: E^* \to [-\infty, \infty]$, is defined as:
\begin{align}
f^*(s') = \sup_{x \in E} \langle x, s' \rangle - f(x) \label{eq:f*}
\end{align}

\begin{example}
The entropy functions defined in previous examples have the following convex conjugates:
\begin{align}
f_{KL}^*(s') &= \exp(s') - 1 \label{eq:f_KL*} \\
f_{TV}^*(s') &= \begin{cases}
    \max(s', -1) & \text{if } s' \leq 1 \\
    \infty & \text{if } s' > 1 
\end{cases} \label{eq:f_tv*} \\
f_{PTV}^*(s') &= \max(s', -1) \label{eq:f_PTV*} \\
\iota^*_{\{0\}}(s') &= s' \label{f_=*} \\
\bold{0}^*(s') &= \iota_{\{0\}} \label{f_0*} \\
\iota_{[0,1]}^*(s') &= \begin{cases}
    0 & \text{if } s' \leq 0 \\
    s' & \text{if } s' \geq 0 
\end{cases} \label{eq:f_[0,1]*}
\end{align}
\end{example}

\begin{example}
Let $E = L_1(X, dx)$ and $E^* = L_\infty(X, dx)$. 
Consider the mapping 
$$E \ni r \mapsto D_{f}(r  dx \parallel \mu) := \langle f(\frac{r}{d\mu/dx}), \mu \rangle + f'_\infty \| (r  dx)^\perp \| \in \mathbb{R} \cup \{\infty\},$$
where $f$ is an entropy function. 

According to \cite[Theorem A.3]{chizat2018scaling}, this mapping is a proper convex, weakly lower-semi continuous function in $L_1(X, dx)$. Its conjugate is given by 
\begin{align}
E^* \ni \phi \mapsto D_{f}^*(\phi  dx \parallel \mu) := \langle f^*(\phi), \mu \rangle + \langle \iota_{\leq f_\infty'} \phi, dx \rangle . \label{eq:Df*}
\end{align}

Thus, we have the following:
\begin{itemize}
    \item Let $E = L_1(\Omega^2, dxdy)$ and $E^* = L_\infty(\Omega^2, dxdy)$, we have: 
    \begin{align}
     KL^*(\phi  dxdy \parallel K  dxdy) = \langle (e^{\frac{\phi}{\epsilon}} - 1)K, dxdy \rangle \label{KL*}
    \end{align}
    
    \item Let $E = L_1(\Omega, dx)$ and $E^* = L_\infty(\Omega, dx)$. Consider the mapping: 
    \begin{align}
     r \mapsto TV(\lambda_1r dx \parallel \lambda_1  \mu).
    \end{align}
    First, it is proper convex since 
    $$r' \mapsto TV(r'  dx \parallel \lambda_1  \mu)$$ 
    is a proper convex function, and 
    $$E_1 \ni r \mapsto r' = \lambda_1  r \in E_1$$ 
    is a linear operator. 
    
    By Lemma \ref{lem:TV*}, its conjugate, denoted as $TV^*(\lambda_1  \phi  dx \parallel \lambda_1  \mu)$, is given by:
    \begin{align}
    TV^*(\phi  \lambda_1  dx \parallel \lambda_1  \mu) = \begin{cases}
    \langle \max(\phi, -\lambda_1), \mu \rangle & \text{if } \phi \leq \lambda_1,  dx\text{-a.s.};  \phi = 0,  dx\mid_{\lambda_1=0}\text{-a.s.} \\ 
    \infty & \text{elsewhere}
    \end{cases} \label{eq:TV*}
    \end{align}

    Similarly, replace $TV$ with $PTV$, we have: 
    \begin{align}
    PTV^*(\phi  \lambda_1  dx \parallel \lambda_1  \mu) = \begin{cases}
    \langle \max(\phi, -\lambda_1), \mu \rangle & \text{if } \phi = 0,  dx\mid_{\lambda_1=0} \\ 
    \infty & \text{elsewhere}
    \end{cases} \label{eq:PTV*_0}    
    \end{align}
    In addition, by Lemma \ref{lem:PTV*}, we have: 
    \begin{align}
    \overline{PTV}^*(\phi  \lambda_1  dx \parallel \lambda_1  \mu) = \langle \max(\phi, -\lambda_1), \mu \rangle \label{eq:PTV*}
    \end{align}

    \item Let $E = L_1(\Omega, dx)$ and $E^* = L_\infty(\Omega, dx)$. We have: 
    \begin{align}
    \iota^*_{=}(\phi  dx \parallel \mu) &= \langle \phi, \mu \rangle \label{eq:D_=*}\\ 
    D_{\bold{0}}^*(\phi  dx \parallel \mu) &= \begin{cases}
        0 & \text{if } \phi \leq 0,  dx\text{-a.s.}\\
        \infty & \text{elsewhere}
    \end{cases} \label{eq:D_+*}\\
    \iota^*_{\leq}(\phi  dx \parallel \mu) &= \langle \max(\phi, 0), \mu \rangle \label{eq:D_<=*}
    \end{align}
\end{itemize}
\end{example}

\begin{lemma}\label{lem:TV*}
We claim \eqref{eq:TV*} holds.
\end{lemma}

\begin{proof}
From \eqref{eq:Df*}, we have for each $\phi' \in L_\infty(\Omega, dx)$:  
\begin{align}
TV^*(\phi'  dx \parallel \lambda_1  \mu) := \max_{r} \left( \langle \lambda_1  r  \phi', dx \rangle - TV(\lambda_1  r  dx \parallel \lambda_1  \mu) \right) = \langle \max(\phi', -1), \lambda_1  \mu \rangle + \langle \iota_{\ge 1}(\phi'), dx \rangle \label{pf:TV*_lambda1}
\end{align}

Let $\Omega_1 = \{x: \phi > 0\}$ and $\Omega_2 = \{x: \phi = 0\}$. For each function $r \in E$, let $r \mid_{\Omega_1}, r \mid_{\Omega_2}$ denote the restriction of $r$ on $\Omega_1$ and $\Omega_2$ respectively, and the restriction of measure is defined similarly.

We have: 
\begin{align}
TV^*(\lambda_1  \phi  dx \parallel \lambda_1  \mu) &= \max_{r} \left( \langle r  \phi, dx \rangle - TV(\lambda_1  r  dx \parallel \lambda_1  \mu) \right) \nonumber\\
&= \max_{r} \left( \langle (r \mid_{\Omega_1} + r \mid_{\Omega_2})  \phi, dx \rangle - TV(\lambda_1  r \mid_{\Omega_1}  dx \parallel \lambda_1  \mu) \right) \nonumber\\
&= \max_{r \mid_{\Omega_1}} \left( \langle r \mid_{\Omega_1}  \phi, dx \rangle - TV(\lambda_1  r \mid_{\Omega_1} dx \parallel \lambda_1\mu ) \right) + \max_{r \mid_{\Omega_2}} \langle r \mid_{\Omega_2}  \phi, dx  \rangle \nonumber\\
&= \underbrace{{TV}^* \left( \frac{\phi \mid_{\Omega_1}}{\lambda_1 \mid_{\Omega_1}}  dx \parallel \lambda_1 \mu \right)}_{B_1} + \underbrace{\max_{r \mid_{\Omega_2}} \langle r \mid_{\Omega_2}  \phi, dx \rangle}_{B_2} \nonumber
\end{align}

where we use the convention $\frac{0}{0} = 1$ for the notation $\frac{\phi \mid_{\Omega_1}}{\lambda_1\mid_{\Omega_1}}$.

It is straightforward to verify $B_1, B_2 \ge -\infty$.

Case 1: If $dx \mid_{\Omega_2}(\{\phi \neq 0\}) > 0$, then $B_2 = \infty$.

Case 2: If $dx \mid_{\Omega_1}(\{\phi > \lambda_1\}) > 0$, then we have: 
$$B_1 \ge \langle \iota_{\ge 1} \frac{\phi \mid_{\Omega_1}}{\lambda_1 \mid_{\Omega_1}}, dx \mid_{\Omega_1} \rangle - \| \lambda_1  \mu \| = \infty.$$

Case 3: If $\phi \leq \lambda_1$, $dx$-a.s., and $\phi = 0$, $dx \mid_{\lambda_1=0}$-a.s., then: 
\begin{align}
    B_1 = \langle \max(\phi, -\lambda_1), \mu \rangle,  B_2 = 0
\end{align}
Thus, we complete the proof. 
\end{proof}

Next, we introduce the sub-differential operator of a convex function $f$:  
\begin{align}
\partial f(x)=\{y\in E^*: \langle x'-x,y \rangle \leq f(x')-f(x), \forall x' \in E \}
\end{align}

\begin{lemma}\label{lem:product_space}
Given topological pairs $(E_1, E_1^*)$ and $(E_2, E_2^*)$, it is a well-known result that $(E_1 \times E_2, E_1^* \times E_2^*)$ defines a topological pair with the pairing 
$$(E_1 \times E_2, E_1^* \times E_2^*) \ni \left((x_1, x_2), (x_1^*, x_2^*)\right) \mapsto \langle x_1, x_1^* \rangle + \langle x_2, x_2^* \rangle.$$

If $E_1 \ni x_1 \mapsto f_1(x_1)$ and $E_2 \ni x_2 \mapsto f_2(x_2)$ are two functions, then $(f_1 \oplus f_2)^* = f_1^* \oplus f_2^*$. In addition, if $f_1$ and $f_2$ are convex functions, then $f_1 \oplus f_2$ is convex.  
\end{lemma}

The proof directly follows from the definition of convexity and conjugate functions. 

\begin{lemma}\label{lem:PTV*}
We claim that the function $r \mapsto \overline{D}(\lambda_1 r  dx \parallel \lambda_1  \mu)$ is a proper, lower-semi continuous, convex function. 
In addition, $\overline{PTV}^*$ is given by \eqref{eq:PTV*}.     
\end{lemma}

\begin{proof}
We can split the space $\Omega$ into $\Omega_1 = \{x: \lambda_1 > 0\}$ and $\Omega_2 = \{x: \lambda_1 = 0\}$. 
Then,
\begin{align}
\overline{PTV}(\lambda_1 r  dx \parallel \lambda_1  \mu) &= PTV(\lambda_1 r  dx \parallel \lambda_1  \mu) + \iota_\leq(r  dx \parallel \mu) \nonumber\\
&= \underbrace{PTV(\lambda_1 r\mid_{\Omega_1} dx \parallel \lambda_1 \mu \mid_{\Omega_1})}_{A_1} + \underbrace{\iota_{\leq}(r  \mid_{\Omega_2}dx  \parallel \mu \mid_{\Omega_2})}_{A_2} \nonumber 
\end{align}
The mappings $r \mid_{\Omega_1} \mapsto A_1$ and $r \mid_{\Omega_2} \mapsto A_2$ are proper, convex, weakly lower semi-continuous functions. Thus, $\overline{PTV}(\lambda_1 r  dx \parallel \lambda_1  \mu)$ is proper, convex, and weakly lower semi-continuous. 

In addition, its conjugate is given by 
\begin{align}
\overline{PTV}^*(\phi \lambda_1  dx \parallel \lambda_1  \mu) &= PTV^*(\lambda_1 \phi \mid_{\Omega_1} \parallel \lambda_1  dx \mid_{\Omega_1}) + \iota_{\leq}^*(\phi \mid_{\Omega_2} \parallel \mu \mid_{\Omega_2}) \nonumber\\
&= \int_{\Omega_1} \max(\phi, -\lambda_1)  d\mu + \int_{\Omega_2} \max(\phi, 0)  d\mu \nonumber\\
&= \langle \max(\phi, -\lambda_1), \mu \rangle \nonumber  
\end{align}
and this completes the proof. 
\end{proof}

Given topological pairs $(E_1, E^*_1)$ and $(E_2, E^*_2)$, suppose $A: E_1 \to E_2$ is a continuous linear operator. Its dual operator $A^*: E^*_2 \to E_1^*$ is a continuous linear operator such that 
$$\langle A s_1, s^*_2 \rangle = \langle s_1, A^* s^*_2 \rangle, \forall s_1 \in E_1, s^*_2 \in E^*_2. $$

\begin{example}\label{ex:A}
Consider $(E_1 = L_1(\Omega^2, dxdy), E_1^* = L_\infty(\Omega^2, dxdy))$ and $(E_2 = L_1(\Omega, dx) \times L_1(\Omega, dy), E^*_2 = L_\infty(\Omega, dx) \times L_\infty(\Omega, dy))$. 
For $(E_2, E_2^*)$, the topological pairing is defined by: 
\begin{align}
\left( \left( \frac{d\gamma_1}{dx}, \frac{d\gamma_2}{dy} \right), (\phi, \psi) \right) := \left\langle \frac{d\gamma_1}{dx}, \phi  dx \right\rangle + \left\langle \frac{d\gamma_2}{dy}, \psi  dy \right\rangle.
\end{align}

The following is a linear continuous mapping:
$$E_1 \ni \frac{d\gamma}{dxdy} \to A\left(\frac{d\gamma}{dxdy}\right) = \left(\frac{\gamma_1}{dx}, \frac{\gamma_2}{dy}\right) \in E_2,$$
then its dual operator $A^*$ is defined by:
$$
E_2^* \ni (\phi, \psi) \to A^*(\phi, \psi) = \phi \oplus \psi \in E_1^*.
$$
\end{example}

Next, we introduce the following celebrated theorem in convex analysis:
\begin{theorem}[Rockafellar Theorem,  [Theorem 3, \cite{rockafellar1967duality}, [Theorem A.1, \cite{chizat2018scaling}]\label{thm:dual}
Given topological pairs $(E_1, E_1^*)$ and $(E_2, E_2^*)$, a linear continuous mapping $A: E_1 \to E_2$ with its dual operator $A^*: E_2^* \to E_1$. Suppose $F: E_1 \to [-\infty, \infty]$ and $G: E_2 \to [-\infty, \infty]$ are proper, lower semi-continuous functions. In addition, there exists $x \in \text{Dom}(F)$ such that $G$ is continuous at $Ax$. 

Then the following strong duality holds: 
\begin{align}
\inf_{x_1 \in E_1} F(x_1) + G(Ax_1) = \sup_{x^*_2 \in E_2^*} -F^*(A^* x_2^*) - G^*(-x_2^*).\label{eq:duality}
\end{align}
Moreover, the infimum is attained by some $x_1$. If $x_2^* \in E_2^*$ is a maximizer, there exists a minimizer $x_1 \in E_1$ such that 
$Ax_1 \in \partial G^*(-x_2^*)$ and $A^* x_2^* \in \partial F(x_1)$.
\end{theorem}

\begin{remark}\label{rk:dual_form}
Now we discuss the derivation of \eqref{eq:egopt_dual}. 
Apply $(E_1, E_1^*), (E_2, E_2^*)$, operator $A$, defined in Example \ref{ex:A}, set $F, G$ as 
\begin{align}
E_1 \ni \frac{d\gamma}{dxdy} &\mapsto F\left(\frac{d\gamma}{dxdy}\right):= KL(\gamma \parallel Kdxdy) \nonumber \\
E_2 \ni \left(\frac{d\gamma_1}{dx}, \frac{d\gamma_2}{dy}\right) &\mapsto G\left(\frac{d\gamma_1}{dx}, \frac{d\gamma_2}{dy}\right):= \begin{cases}
TV(\lambda_1 \gamma_1 \parallel \lambda_1 \mu)\\
\overline{PTV}(\lambda_1 \gamma_1 \parallel \lambda_1 \mu)
\end{cases} + 
\begin{cases}
TV(\lambda_2 \gamma_2 \parallel \lambda_2 \nu)\\
\overline{PTV}(\lambda_2 \gamma_2 \parallel \lambda_2 \nu)
\end{cases}.
\end{align}

Then, by lemma \ref{lem:product_space} and the convex conjugate functions \eqref{eq:TV*}, \eqref{eq:PTV*}, we have the convex conjugate functions $F^*, G^*$ as: 
\begin{align}
E_1^* \ni \Phi \mapsto F^*(\Phi) &= \langle (e^{-\Phi/\epsilon} - 1) K, dxdy \rangle \nonumber \\
E_2^* \ni (\phi, \psi) \mapsto G^*(\phi, \psi) &= \begin{cases}
TV^*(\lambda_1 \phi \parallel \lambda_1 \mu) \\
PTV^*(\lambda_1 \phi \parallel \lambda_1 \mu)  
\end{cases} + \begin{cases}
TV^*(\lambda_2 \psi \parallel \lambda_2 \nu) \\
PTV^*(\lambda_2 \psi \parallel \lambda_2 \nu)      
\end{cases} \\
&= \langle \max(\phi, -\lambda_1), \mu \rangle + \langle \max(\psi, -\lambda_2), \nu \rangle + \begin{cases}
\iota_{\phi \in \mathcal{S}^-(-\lambda_1, dx)} \\
0
\end{cases} + \begin{cases}
\iota_{\psi \in \mathcal{S}^-(-\lambda_2, dy)} \\
0
\end{cases} \nonumber
\end{align}
where 
$$\mathcal{S}^-(-\lambda_1, dx) := \{\phi \ge -\lambda_1, dx-\text{a.s.}; \phi = 0, dx \mid_{\lambda_1=0} - \text{a.s.}\},$$
and $\mathcal{S}^-(-\lambda_2, dy)$ is defined similarly. 
\end{remark}

\subsubsection{Alternative dual optimization}
By \cite[Theorem 3.1]{chizat2018scaling} (or Theorem \ref{thm:dual}), suppose $(\phi, \psi)$ is a maximizer for the dual problem \eqref{eq:egopt_dual}, then the corresponding minimizer can be obtained by $\gamma = e^{\frac{\phi}{\epsilon}} K e^{\frac{\psi}{\epsilon}} dxdy$. 

It remains to find optimal dual variations $\phi,\psi$ in the dual problem \eqref{eq:egopt_dual}. Given $\psi$, we will find the optimal $\phi$. 
First, we set $u = e^{\phi/\epsilon}, v = e^{\psi/\epsilon}, Kv = \langle K(x, \cdot)v(\cdot), dy\rangle, K^\top u = \langle K(\cdot, y) u(\cdot), dx \rangle$.

We choose the $TV(\lambda_1 \gamma_1 \parallel \lambda_1 \mu)$ term, then optimizing $\phi$ is equivalent to  
\begin{align}
\max_{\phi \ge -\lambda_1, dx-\text{a.s.}} -\epsilon \langle e^{\frac{\phi}{\epsilon}} K, dxdy \rangle + \langle \min(\phi, \lambda_1), \mu \rangle. \nonumber 
\end{align}

By the first-order derivative method, combined with the fact that $\phi \ge -\lambda_1, dx-\text{a.s.}$ and $\phi = 0, dx \mid_{\lambda_1=0} - \text{a.s.}$, we obtain the optimal $\phi$ (or $u$) given by: 
\begin{align}
u(x) &= \exp(\phi/\epsilon) = \frac{\overbrace{\text{clip}\left(\frac{d\mu}{dx}, [e^{-\lambda_1/\epsilon}, e^{\lambda_1/\epsilon}] Kv \right)}^{\text{prox}_{TV(\lambda_1\cdot|\lambda_1\mu)/\epsilon}^{KL}(Kv)}(x)}{Kv(x)} = \text{clip}\left(\frac{d\mu/dx}{Kv(x)}, [e^{-\lambda_1/\epsilon}, e^{\lambda_1/\epsilon}] \right), \label{eq:tv_update_u}
\end{align}
where $\text{clip}(a, [a_1, a_2]) = \begin{cases}
a & \text{if } a \in [a_1, a_2] \\ 
a_1 & \text{if } a < a_1 \\ 
a_2 & \text{if } a > a_2
\end{cases}$, and $\text{prox}_{TV}^{KL}$ is called the \textbf{proximal operator} 
for the KL divergence and the functional $\frac{1}{\epsilon} TV(\lambda_1 \cdot \parallel \lambda_1 \mu )$: 
\begin{align}
\text{prox}^{KL}_{F_1/\epsilon}(z):= \arg\min_{s} F_1(s) + \epsilon KL(s|z). \nonumber  
\end{align}
In addition, the operator  
$$Kv \mapsto \text{proxdiv}^{KL}_{F_1/\epsilon}(Kv):= \frac{\text{prox}_{F_1/\epsilon}^{KL}(Kv)}{Kv}$$
is called the \textbf{proximal-divide operator}. 

Similarly, when we choose the $\overline{PTV}(\gamma_1 \parallel \mu)$ term, $\phi$ (or $u = e^{\phi/\epsilon}$) is updated by the corresponding proximal-divide operator:
\begin{align}
u = \text{proxdiv}_{ \overline{PTV}(\lambda_1 \cdot\parallel \lambda_1 \mu)/\epsilon}^{KL}(Kv) = \min\left(\frac{d\mu/dx}{Kv}, e^{\lambda_1/\epsilon}\right). \label{eq:ptv_update_u}
\end{align}

Similarly, for fixed $\phi$ (or $u$), the optimal $\psi$ (or $v$) is given by 
\begin{align}
v &= \exp(\psi/\epsilon) = \text{proxdiv}^{KL}_{TV(\lambda_2\cdot|\lambda_2 \nu)/\epsilon}(K^\top u) = \text{clip}\left(\frac{\nu/dy}{K^\top u}, [e^{-\lambda_2/\epsilon}, e^{\lambda_2/\epsilon}] \right), &\text{for } TV(\lambda_2 \gamma_2\parallel \lambda_2 \nu) \label{eq:tv_update_v} \\
v &= \exp(\psi/\epsilon) = \text{proxdiv}^{KL}_{\overline{PTV}(\lambda_2\cdot|\lambda_2 \nu)/\epsilon}(K^\top u) = \min\left(\frac{\nu/dy}{K^\top u}, e^{\lambda_2/\epsilon} \right), &\text{for } \overline{PTV}(\lambda_2 \gamma_2\parallel \lambda_2 \nu) \label{eq:ptv_update_v}.
\end{align}
  
\subsection{Relation between GOPT and classical OT problem}

It has been studied that the classical OPT problem \eqref{eq:opt_1}, \eqref{eq:opt_2} is equivalent to a classical balanced OT problem. By this relation, one can apply the linear programming solver in OT for the OPT problem. In this section, we extend this relation to GOPT \eqref{eq:gopt}. 

In particular, we choose $\overline{PTV}(\lambda_1 \gamma_1\parallel \lambda_1 \mu)$ and $\overline{PTV}(\lambda_2 \gamma_2\parallel \lambda_2 \nu)$ in GOPT, and we obtain:
\begin{align}
GOPT(\mu, \nu) = \inf_{\gamma \in \Gamma_\leq (\mu, \nu)} \langle c - (\lambda_1 \oplus \lambda_2), \gamma \rangle + \underbrace{\langle \lambda_1, \mu \rangle + \langle \lambda_2, \nu \rangle}_{\text{constant}} \label{eq:gopt_1}. 
\end{align}

Introduce an auxiliary point $\hat\infty$ and let $\hat\Omega = \Omega \cup \{\hat\infty\}$. 
We define $\hat{c} : \hat\Omega^2 \to \mathbb{R}$ with
\begin{align}
    \hat{c}(x, y) :=
    \begin{cases}
    c(x, y) - (\lambda_1(x) + \lambda_2(y)) & \text{if } (x, y) \in \Omega^2\\ 
    0 & \text{elsewhere}
    \end{cases}
\end{align}
In addition, define 
\begin{align}
\hat{\mu} &= \mu + |\nu| \delta_{\hat\infty} \nonumber\\
\hat{\nu} &= \nu + |\mu| \delta_{\hat\infty} \nonumber
\end{align}
Thus, $\|\hat{\mu}\| = \|\hat{\nu}\|$. By \cite{caffarelli2010free, bai2022sliced}, the following mapping: 
\begin{align}
\Gamma_\leq (\mu, \nu) \ni \gamma \mapsto \hat{\gamma} := \gamma + (\mu - \gamma_1) \otimes \delta_{\hat\infty} + \delta_{\hat\infty} \otimes (\nu - \gamma_2) + |\gamma| \delta_{(\hat\infty, \hat\infty)} \in \Gamma(\hat{\mu}, \hat{\nu}) \label{eq:gamma_gamma_hat} 
\end{align}
is a well-defined bijection. 

We claim the following:
\begin{lemma}\label{lem:opt_ot}
Consider the balanced OT problem:
\begin{align}
\inf_{\hat{\gamma} \in \Gamma(\hat{\mu}, \hat{\nu})} \int_{\hat{\Omega}^2} \hat{c}  d\hat{\gamma} \label{eq:ot_hat}  
\end{align}
Choose $\gamma \in \Gamma_\leq (\mu, \nu)$ and find $\hat{\gamma}$ by \eqref{eq:gamma_gamma_hat}. 
We have that $\hat{\gamma}$ is a minimizer of \eqref{eq:ot_hat} if and only if $\gamma$ is a minimizer for \eqref{eq:gopt_1}.  
\end{lemma}
\begin{proof}
We have:
\begin{align}
\int_{\hat{\Omega}^2} \hat{c}  d\hat{\gamma} = \int_{\Omega^2} (c - (\lambda_1 \oplus \lambda_2))  d\gamma \nonumber 
\end{align}
Since the mapping \eqref{eq:gamma_gamma_hat} is a bijection, we have that $\hat{\gamma}$ is a minimizer for the OT problem \eqref{eq:ot_hat} if and only if $\gamma$ is a minimizer for the OPT problem \eqref{eq:gopt_1}.
\end{proof}

In addition, it is straightforward to verify the following:
\begin{lemma}\label{lem:large_lambda}
In problem \eqref{eq:gopt_1}, choose $\epsilon' > 0$, then for each optimal $\gamma \in \Gamma_\leq (\mu, \nu)$, we have $$\gamma(\{c - (\lambda_1 \oplus \lambda_2) \ge \epsilon'\}) = 0.$$
\end{lemma}
\begin{proof}
Pick an optimal $\gamma \in \Gamma_\leq (\mu, \nu)$ and let $\mathcal{A} = \{c - (\lambda_1 \oplus \lambda_2) \ge \epsilon'\}$. Suppose $\gamma(\mathcal{A}) > 0$. 

We have $\gamma \mid_{\mathcal{A}} \leq \gamma$, thus $\gamma \mid_{\mathcal{A}^c} \in \Gamma_\leq (\mu, \nu)$. In addition,
\begin{align}
\langle c - (\lambda_1 \oplus \lambda_2), \gamma \rangle - \langle c - (\lambda_1 \oplus \lambda_2), \gamma \mid_{\mathcal{A}^c} \rangle
= \langle c - (\lambda_1 \oplus \lambda_2), \gamma \mid_{\mathcal{A}} \rangle \ge \epsilon' \|\gamma \mid_{\mathcal{A}}\| > 0 \nonumber  
\end{align}
This is a contradiction since $\gamma$ is optimal. 
\end{proof}

\begin{lemma}\label{lem:large_lambda_2}
In problem \eqref{eq:gopt_1}, suppose $|\nu| \leq |\mu|$. If for all $(x, y) \in \Omega^2$, $c - (\lambda_1 \oplus \lambda_2) < -\epsilon_1$ for some $\epsilon_1 > 0$, then for each optimal transportation plan $\gamma$, we have $\gamma_2 = \nu$.
\end{lemma}
\begin{proof}
Choose an optimal $\gamma$ and suppose $|\gamma| < |\nu|$. Then by [Lemma E.1 \cite{bai2024efficient}], there exists a $\gamma'\in\Gamma(\mu,\nu)$ such that $\gamma'_2 = \nu$, $\gamma \leq \gamma'$, and $\gamma'_1 \leq \mu$. Then we have:
\begin{align}
\langle c - (\lambda_1 \oplus \lambda_2), \gamma' \rangle - \langle c - (\lambda_1 \oplus \lambda_2), \gamma \rangle 
= \langle c - (\lambda_1 \oplus \lambda_2), \gamma' - \gamma \rangle 
\leq -\epsilon_1 \|\gamma' - \gamma\| < 0 \nonumber
\end{align}
This is a contradiction to the fact that $\gamma$ is an optimal transportation plan. 
\end{proof}

\subsection{Algorithms for GOPT problems}
\subsubsection{OPT problems in a discrete setting}

Suppose $\mu = \sum_{i=1}^n p_i \delta_{x_i}$ and $\nu = \sum_{j=1}^m q_j \delta_{y_j}$, where $p > 0_n$ and $q > 0_m$ for all $i,j$. We set $dx = \frac{1}{n} \sum_{i=1}^n \delta_{x_i}$ and $dy = \frac{1}{m} \sum_{j=1}^m \delta_{y_j}$. Without loss of generality, we can redefine $dx$ and $dy$ as $dx = \sum_{i=1}^n \delta_{x_i}$ and $dy = \sum_{j=1}^m \delta_{y_j}$ since the (positive) scaling of reference measures will not change the OPT problems. In addition, we set $\lambda_1 \in \mathbb{R}_{+}^n$ and $\lambda_2 \in \mathbb{R}_{+}^m$.

Finally, choose $\epsilon > 0$. We adopt $c \in \mathbb{R}^{n \times m}$ as the cost function (matrix) in OPT problems and define $K = e^{-c/\epsilon}$. The negative entropy is defined by  
$-H(\gamma) := \sum_{i,j} (\ln(\gamma_{ij}) - 1) \gamma_{ij}$.

Then the generalized OPT \eqref{eq:gopt}, the entropic version \eqref{eq:egopt}, and the corresponding dual forms \eqref{eq:egopt_dual}, can be written as follows: 
\begin{align}
GOPT(\mu, \nu; \lambda_1, \lambda_2) = \min_{\gamma \in \mathbb{R}_+^{n \times m}} \sum_{i,j} c_{ij} \gamma_{ij}
&+ \begin{cases}
    \sum_i (\lambda_1)_i |(p_i - (\gamma 1_m)_i)| \\
    \sum_i (\lambda_1)_i (p_i - (\gamma 1_m)_i) + \iota_{\{\gamma 1_m \leq p\}}
\end{cases} \nonumber \\
&+ \begin{cases}
    \sum_j (\lambda_2)_j |(q_j - (\gamma^\top 1_n)_j)| \\
    \sum_j (\lambda_2)_j (q_j - (\gamma^\top 1_n)_j) + \iota_{\{\gamma^\top 1_n \leq q\}}
\end{cases}
\label{eq:gopt_discrete}
\end{align}

\begin{align}
EGOPT(\mu, \nu; \lambda_1, \lambda_2, \epsilon) = \min_{\gamma \in \mathbb{R}_+^{n \times m}} \sum_{i,j} c_{ij} \gamma_{ij} + \epsilon \sum_{i,j} (\ln(\gamma_{ij}) - 1) \gamma_{ij}
&+ \begin{cases}
    \sum_i (\lambda_1)_i |(p_i - (\gamma 1_m)_i)| \\
    \sum_i (\lambda_1)_i (p_i - (\gamma 1_m)_i) + \iota_{\{\gamma 1_m \leq p\}}
\end{cases} \nonumber \\
&\hspace{-3em}+ \begin{cases}
    \sum_j (\lambda_2)_j |(q_j - (\gamma^\top 1_n)_j)| \\
    \sum_j (\lambda_2)_j (q_j - (\gamma^\top 1_n)_j) + \iota_{\{\gamma^\top 1_n \leq q\}}
\end{cases}
\label{eq:egopt_discrete}
\end{align}

\begin{align}
Dual-EGOPT(\mu, \nu; \lambda_1, \lambda_2, \epsilon) = \max_{\substack{\phi \in \mathbb{R}^n \\ \psi \in \mathbb{R}^m}} -\epsilon \sum_{i,j} (e^{\frac{\phi_i + \psi_j}{\epsilon}} - 1) K_{ij}
&+ \begin{cases}
    \sum_i \min((\lambda_1)_i, \phi_i) p_i \\
    \sum_i \min((\lambda_1)_i, \phi_i) p_i + \iota_{\{\phi \ge -\lambda_1\}}
\end{cases} \nonumber \\
&+ \begin{cases}
    \sum_j \min((\lambda_2)_j, \psi_j) q_j \\
    \sum_j \min((\lambda_2)_j, \psi_j) q_j + \iota_{\{\psi \ge -\lambda_2\}}
\end{cases}
\label{eq:egopt_dual_discrete}
\end{align}

\subsubsection{Sinkhorn algorithm}\label{sec:sinkhorn}

The Sinkhorn algorithm aims to solve \eqref{eq:egopt_dual_discrete}. In particular, set $u = e^{\phi/\epsilon} \in \mathbb{R}^n$ and $v = e^{\psi/\epsilon} \in \mathbb{R}^m$. From the fact $dx \ll \mu$ in the discrete setting, the alternative optimization steps \eqref{eq:tv_update_u} (or \eqref{eq:ptv_update_u}), \eqref{eq:tv_update_v} (or \eqref{eq:ptv_update_v}) become:
\begin{align}
&u \gets \text{proxdiv}^{KL}_{TV(\lambda_1\cdot|\lambda_1p)/\epsilon}(Kv) = \text{clip}\left(\frac{p}{Kv}, [e^{-\lambda_1/\epsilon}, e^{\lambda_1/\epsilon}]\right) &\text{for } TV(\lambda_1 \gamma_1\parallel \lambda_1 \mu) \label{eq:tv_update_u_discrete}\\
&u \gets \text{proxdiv}^{KL}_{\overline{PTV}(\lambda_1\cdot|\lambda_1p)/\epsilon}(Kv) = \min\left(\frac{p}{Kv}, e^{\lambda_1/\epsilon}\right) &\text{for } \overline{PTV}(\lambda_1 \gamma_1\parallel \lambda_1 \mu) \label{eq:ptv_update_u_discrete}\\
&v \gets \text{proxdiv}^{KL}_{TV(\lambda_2\cdot|\lambda_2q)/\epsilon}(K^\top u) = \text{clip}\left(\frac{q}{K^\top u}, [e^{-\lambda_2/\epsilon}, e^{\lambda_2/\epsilon}]\right) &\text{for } TV(\lambda_2 \gamma_2\parallel \lambda_2 \nu) \label{eq:tv_update_v_discrete}\\
&v \gets \text{proxdiv}^{KL}_{\overline{PTV}(\lambda_2\cdot|\lambda_2q)/\epsilon}(K^\top u) = \min\left(\frac{q}{K^\top u}, e^{\lambda_2/\epsilon}\right) &\text{for } \overline{PTV}(\lambda_2 \gamma_2\parallel \lambda_2 \nu) \label{eq:ptv_update_v_discrete}
\end{align}
where ``proxdiv'' denotes the ``\textbf{proximal-divide operator}''.

In summary, the Sinkhorn algorithm is presented in Algorithm \ref{alg:sinkhorn}. When solving the (entropic) GOPT problem, if we choose $TV(\lambda_1 \gamma_1\parallel \lambda_2 \mu)$ as the first penalty term, we apply \eqref{eq:tv_update_u_discrete} as $\text{proxdiv}_{F_1}$; if we choose $\overline{PTV}(\lambda_1 \gamma_1\parallel \lambda_2 \mu)$ as the first penalty term, we apply \eqref{eq:ptv_update_u_discrete} as $\text{proxdiv}_{F_1}$. Similarly, we set $\text{proxdiv}_{F_2}$ accordingly.

\begin{algorithm}\caption{opt-Sinkhorn}
\label{alg:sinkhorn}
\KwInput{$c, \epsilon, \text{proxdiv}_{F_1}, \text{proxdiv}_{F_2}$}
\KwOutput{$\gamma$}
Initialize $v = 1_m$, $K = e^{-c/\epsilon}$\\
\For{$l=1,2,\ldots$}{
$u = \text{proxdiv}^{KL}_{F_1}(Kv)$ by \eqref{eq:tv_update_u_discrete} or \eqref{eq:ptv_update_u_discrete}\\
$v = \text{proxdiv}^{KL}_{F_2}(K^\top u)$ by \eqref{eq:tv_update_v_discrete} or \eqref{eq:ptv_update_v_discrete}\\
\text{If $(u,v)$ converge}, {break}
}
$\gamma \gets (u_i K_{ij} v_j)_{ij}$
\end{algorithm}

\subsubsection{Linear programming}\label{sec:lp}
When we select $PTV$ as the first and second penalty terms, based on the relation between GOPT and OT as discussed in Lemma \eqref{lem:opt_ot}, we proceed as follows:
Let $\hat{c} \in \mathbb{R}^{(n+1) \times (m+1)}$ with 
\begin{align}
\hat{c}_{i,j} = \begin{cases}
    c_{i,j} - (\lambda_1)_i - (\lambda_2)_j & \text{if } i \in [1:n], j \in [1:m] \\
    0 & \text{elsewhere}
\end{cases} \label{alg:c_hat}
\end{align}
Similarly, set $\hat{p} \in \mathbb{R}^{n+1}, \hat{q} \in \mathbb{R}^{m+1}$ with
\begin{align}
\hat{p}_i &= \begin{cases}
    p_i & \text{if } i \in [1:n] \\
    \sum_j q_j & \text{if } i = n+1
\end{cases} \label{alg:p_hat} \\
\hat{q}_j &= \begin{cases}
    q_j & \text{if } j \in [1:m] \\
    \sum_i p_i & \text{if } j = m+1
\end{cases} \label{alg:q_hat}
\end{align}
The linear programming solver is given in Algorithm \ref{alg:lp}.

\begin{algorithm}\caption{opt-lp}
\label{alg:lp}
\KwInput{$c, p, q, \lambda_1, \lambda_2$}
\KwOutput{$\gamma$}
Compute $\hat{c}, \hat{p}, \hat{q}$ using \eqref{alg:c_hat}, \eqref{alg:p_hat}, \eqref{alg:q_hat}\\
Solve the OT problem 
$\hat{\gamma}^* = \arg\min_{\hat{\gamma} \in \Gamma(\hat{p}, \hat{q})} \langle \hat{c}, \hat{\gamma} \rangle$\\
Set $\gamma \gets \hat{\gamma}^*[1:n, 1:m]$
\end{algorithm}

\subsubsection{Extreme case}
Suppose $|\mu| \geq |\nu|$. We consider the following special optimal partial transport problem:
\begin{align}
SOPT(\mu, \nu) &:= \inf_{\substack{\gamma \in \Gamma_\leq(\mu, \nu)\\ \gamma_2 = \nu}} \langle c, \gamma \rangle \label{eq:sopt} \\
&= \inf_{\gamma \in \mathcal{M}_+(\Omega^2)} \langle c, \gamma \rangle + \iota_\leq(\gamma_1 | \mu) + \iota_=(\gamma_2 | \nu) \label{eq:sopt_1}
\end{align}
\begin{remark}
In the discrete case, we can apply the linear programming method \eqref{alg:lp} to solve it. In particular, we set $\lambda_1 \equiv 0$, $\lambda_2 \equiv \max(c) + 1$ in the PGOPT problem \eqref{eq:gopt_1}. By Lemma \ref{lem:large_lambda_2}, the optimal solution satisfies $\gamma_2 = \nu$ and $\gamma_1 \leq \mu$.
\end{remark}

The entropic version of the above Special-OPT (up to a constant) is:
\begin{align}
ESOPT(\mu, \nu) &:= \inf_{\substack{\gamma \in \mathcal{M}_+(\Omega^2)\\ \gamma \ll dxdy}} \epsilon KL(\gamma \parallel K dxdy) + \iota_\leq(\gamma_1 | \mu) + \iota_=(\gamma_2 | \nu) \label{eq:esopt}
\end{align}

By the convex-conjugate of $\iota_=$ and $\iota_\leq$ (see \eqref{eq:D_=*}, \eqref{eq:D_<=*}), we can derive the equivalent dual form:
\begin{align}
Dual-SOPT(\mu, \nu) = \sup_{\substack{\phi \in L_\infty(\Omega, dx)\\ \psi \in L_\infty(\Omega, dy)}} -\epsilon \langle e^{\phi \oplus \psi} K, dxdy \rangle + \langle \min(\phi, 0), \mu \rangle + \langle \psi, \nu \rangle \label{eq:esopt-dual}
\end{align}
\begin{remark}
The problem \eqref{eq:sopt} and its entropic version \eqref{eq:esopt} can be obtained by setting $\lambda_1, \lambda_2$ to extreme values in GOPT problems \eqref{eq:gopt},\eqref{eq:egopt}.

Indeed, for each $D\subset \Omega$, and let $D^c = \Omega \setminus \{D\}$, it is clear that as $\lambda_1\to \infty$ on $D$, and $\lambda_1$ is fixed on $D^c$, we have:
\begin{align}
TV(\lambda_1 \cdot \parallel \lambda_1 \mu) &\to \iota_=(\cdot\mid_D\parallel \mu \mid_D)+ TV(\lambda_1 \cdot |_{D^c} \parallel \lambda_1 \mu |_{D^c}) \label{eq:PTV_infty} \\
\overline{PTV}(\lambda_1 \cdot \parallel \lambda_1 \mu) &\to \iota_=(\cdot |_{D}\parallel \mu |_{D}) + \overline{PTV}(\lambda_1 \cdot |_{D^c} \parallel \lambda_1 \mu |_{D^c}) \label{eq:TV_infty}
\end{align}

Similarly, by setting $\lambda_1 = 0$ on $D$, we obtain:
\begin{align}
TV(\lambda_1 \cdot \parallel \lambda_1 \mu) &= TV(\lambda_1 \cdot |_{D^c} \parallel \lambda_1 \mu |_{D^c}) \\
\overline{PTV}(\lambda_1 \cdot \parallel \lambda_1 \mu) &= \iota_\leq(\cdot, \mu) + \overline{PTV}(\lambda_1 \cdot |_{D^c} \parallel \lambda_1 \mu |_{D^c})
\end{align}

We choose $\overline{PTV}(\lambda_1 \gamma_1\parallel\lambda_1 \mu)$ as the first penalty term and $\overline{PTV}(\lambda_2 \gamma_2\parallel \lambda_2 \nu)$ as the second penalty term. Setting $D = \Omega$, and let $\lambda_1\equiv 0, \lambda_2 \to \infty$. Then we have \eqref{eq:egopt_dual} becomes \eqref{eq:esopt-dual}.
\end{remark}

It is straightforward to verify that the Sinkhorn iteration of \eqref{eq:esopt-dual} in step is
\begin{align}
\begin{cases}
u = e^\phi := \text{proxdiv}_{F_1}(K v) = \min\left(\frac{d\mu/dx}{Kv}, 1_n\right) = \min\left(\frac{p}{Kv}, 1_n\right) \\
v = e^\psi := \text{proxdiv}_{F_2}(K^\top u) = \frac{d\nu/dy}{Kv} = \frac{q}{K^\top u}
\end{cases} \label{alg:esopt_update}
\end{align}

\section{Related Work: Mass-Constrained Partial Transport Problem}\label{sec:mopt}
The mass-constrained optimal partial transport (MOPT) problem is defined by \eqref{eq:mopt}. In this section, we introduce the discrete version and related computations.

First, suppose $\mu = \sum_{i=1}^n p_i \delta_{x_i}$ and $\nu = \sum_{j=1}^m q_j \delta_{y_j}$ with $\eta \in [0, \min(\|p\|, \|q\|)]$, where $\|p\| = \sum_{i=1}^n p_i$ and $\|q\|$ is defined similarly. Then the discrete version of \eqref{eq:mopt} is defined by 
\begin{align}
MOPT(\mu, \nu) := \min_{\gamma \in \Gamma_\leq^\eta(p, q)} \sum_{i,j} c_{ij} \gamma_{ij} \label{eq:mopt_discrete}
\end{align}
where $\Gamma_\leq^\eta(p, q) := \{\gamma \in \mathbb{R}_+^{n \times m}: \gamma 1_n \leq q, \gamma^\top 1_m \leq p, \sum_{ij} \gamma = \eta\}$.

By \cite{chapel2020partial}, it can be solved by the following: 
Let $\hat{c} \in \mathbb{R}^{(n+1) \times (m+1)}$, $\hat{p} \in \mathbb{R}_+^{n+1}$, $\hat{q} \in \mathbb{R}_+^{m+1}$ as 
\begin{align}
\hat{c}_{ij} &= \begin{cases}
    c_{i,j} & \text{if } i \in [1:n], j \in [1:m] \\
    \max(c) + 2\alpha + \beta & \text{if } i = n+1, j = m+1 \\
    \alpha & \text{elsewhere}
\end{cases} \label{eq:c_hat_2} \\
\hat{p}_i &= \begin{cases}
    p_i & \text{if } i \in [1:n] \\
    \|q\| - \eta & \text{if } i = n+1
\end{cases} \hspace{1em}
\hat{q}_j = \begin{cases}
    q_j & \text{if } j \in [1:m] \\
    \|p\| - \eta & \text{if } j = m+1
\end{cases} \label{eq:pq_hat_2}
\end{align}
where $\alpha \ge 0$, $\beta > 0$. 

Then the OT problem 
$$\min_{\hat{\gamma} \in \Gamma(\hat{p}, \hat{q})} \sum_{i,j} \hat{c}_{ij} \hat{\gamma}_{i,j}$$
is equivalent to the MOPT problem \eqref{eq:mopt}. $\hat{\gamma}$ is optimal for the above OT problem if and only if $\gamma = \hat{\gamma}[1:n, 1:m]$ is optimal for the MOPT problem \eqref{eq:mopt}.

Thus, applying $\hat{c}, \hat{p}, \hat{q}$ in Algorithm \eqref{alg:lp} returns the solution for the MOPT problem \eqref{eq:mopt_discrete}.

The entropic regularization version of problem \eqref{eq:mopt_discrete} (up to a constant) is defined as 
\begin{align}
EMOPT(\mu, \nu) := \min_{\gamma \in \Gamma_\leq^\eta(p, q)} KL(\gamma \parallel K) \label{eq:emopt_discrete}
\end{align}
where $K = e^{-c/\epsilon}$ and $\epsilon > 0$.

Note, 
$$\gamma \in \Gamma_\leq^\eta(p, q) = \bigcap_{i=1}^3 \mathcal{C}_i$$
where $\mathcal{C}_1 = \{\gamma \in \mathbb{R}_+^{n \times m}: \gamma 1_m \leq q\}$, $\mathcal{C}_2 = \{\gamma \in \mathbb{R}_+^{n \times m}: \gamma^\top 1_n \leq p\}$, and $\mathcal{C}_3 = \{\gamma \in \mathbb{R}_+^{n \times m}: \sum \gamma = \eta\}$. 
$\mathcal{C}_1$ and $\mathcal{C}_2$ are convex, but not affine sets. $\mathcal{C}_3$ is an affine set.

The entropic MOPT problem \eqref{eq:emopt_discrete} can be solved by the following so-called Dykstra's algorithm \cite{bauschke2000dykstras}: 

\begin{algorithm}\caption{mopt-Dykstra}
\label{alg:mopt}
\KwInput{$p, q, \eta, c$}
\KwOutput{$\gamma$}
\mycommfont{Initialization}\\
\For{i=1,2,3}{
    $\xi^{i} \gets 1_{n \times m}$ \\
}
$\gamma^{(0)} \gets K \frac{m}{\|K\|}$ \\
\mycommfont{Main loop}\\
$k=0$
\For{$l=0,1,2,\ldots$}{
    \For{i=1,2,3}{
        $k \gets k+1$ \\ 
        $\gamma^{(k)} \gets \text{Proj}_{\mathcal{C}_i}^{KL}(\gamma^{(k-1)} \odot \xi^{i})$ \\
        $\xi^{i} \gets \xi^{i} \odot \frac{\gamma^{(k-1)}}{\gamma^{(k)}}$
    }
    Break if $\gamma^{(k)}$ converges
}
\end{algorithm}

where for $i=1,2,3$, 
\begin{align}
\text{Proj}_{\mathcal{C}_i}^{KL}(\gamma) := \min_{\gamma^i \in \mathcal{C}_i} KL(\gamma^i \parallel \gamma) \label{eq:proj}
\end{align}
is called the ``\textbf{Bregman projection}'' (which can be regarded the equivalent primal form of the proximal-divide operator discussed in \eqref{eq:tv_update_u}) onto the convex set $\mathcal{C}_i$. By \cite[Proposition 5]{benamou2015iterative}, these projection operators are given by: 
\begin{align}
\text{Proj}_{\mathcal{C}^1}(\gamma) &= \text{diag} \left( \min \left( \frac{p}{\gamma 1_m}, 1_n \right) \right) \gamma \label{proj_1} \\
\text{Proj}_{\mathcal{C}^2}(\gamma) &= \gamma \text{diag} \left( \min \left( \frac{q}{\gamma^\top 1_n}, 1_m \right) \right) \label{proj_2} \\
\text{Proj}_{\mathcal{C}^3}(\gamma) &= \gamma \frac{\eta}{\|\gamma\|} \label{proj_3}
\end{align}

\section{Acknowledgement}
The author thanks Dr. Soheil Kolouri (soheil.kolouri@vanderbilt.edu) for helpful discussion.

\bibliographystyle{plain}
\bibliography{refs}
\end{document}